\documentclass{amsart}
\usepackage{graphicx} % Required for inserting images
\usepackage{amsmath}
\usepackage{amssymb}
\usepackage{amsthm}

\newtheorem{theorem}{Theorem}[section]
\newtheorem{lemma}[theorem]{Lemma}
\newtheorem{proposition}[theorem]{Proposition}
\newtheorem{corollary}[theorem]{Corollary}

\newtheorem{remark}[theorem]{Remark}
\newtheorem{example}[theorem]{Example}
\numberwithin{equation}{section}
\title{Regularity of the free boundary in an unstable parabolic problem}

\author{Mark Allen}
\address{Department of Mathematics, Brigham Young University, Provo,  UT 84602, US}
\email{allen@mathematics.byu.edu}

\author{Gilles Bokolo-Tamba}
\address{Department of Mathematics, Brigham Young University, Provo,  UT 84602, US}
\email{gilles.bokolo@mathematics.byu.edu}

\begin{document}
\begin{abstract}
    We study the free boundary in an unstable parabolic problem arising from a model in combustion. We consider the physical situation in which the heat advances and prove that the free boundary is a $C^{1,\alpha/2}$ hypersurface. 
\end{abstract}

\maketitle

\section{Introduction}
 We study solutions to the semilinear PDE 
 \begin{equation} \label{e:main}
     u_t - \Delta u = \chi_{\{u>0\}}, 
 \end{equation}
and our main object of study is the free boundary $\Gamma:= \partial \{u>0\}$. The right hand side is discontinuous which makes the study of the problem nontrivial. 
 The equation \eqref{e:main} arises as a limit in a combustion model \cite{ns89}. In one spatial dimension, solutions to \eqref{e:main} were studied in \cite{gh92} where a nontrivial self-similar solution is constructed. More recently, self-similar solutions in dimension $n>1$ were studied in \cite{fg24}. The time-independent version 
 \begin{equation} \label{e:elliptic}
 -\Delta u = \chi_{\{u>0\}}
 \end{equation}
  was studied in \cite{mw07} where it was shown that the free boundary is real analytic except on a singular set of Hausdorff dimension $n-2$. 
  %As in the parabolic setting, for fixed boundary data, the solutions to \eqref{e:elliptic} are not unique. 
  The authors in \cite{mw07} showed that for energy-minimizing solutions, the singular set is empty; consequently, the free boundary for energy-minimizing solutions is locally real analytic. 
  %Having obtained regularity of the free boundary, the authors then proved a posteriori optimal $C^{1,1}$ regularity of the function. At the time of \cite{mw07}, proving regularity of the free boundary before optimal regularity of solutions was a novelty; however, now it is not uncommon to prove regularity of the free boundary first before obtaining optimal regularity of the solution. (cite ??). 
  A non energy minimizing solution of \eqref{e:elliptic} with a cross-singularity for the free boundary was constructed in \cite{aw06}, which also provided an example of a solution to \eqref{e:main} that was not $C^{1,1}$. 

  In one spatial dimension, the authors in \cite{gh92} studied the regularity of the free boundary under the assumption $u_x>0$. In this article we initiate study of the free boundary for the parabolic problem \eqref{e:main} in higher dimensions and include the more difficult situation in which the spatial gradient $\nabla u$ is allowed to vanish. The equation \eqref{e:main} resembles the parabolic obstacle problem 
  $u_t - \Delta u=-\chi_{\{u>0\}}$ which has been studied extensively \cite{acm19}. However, the sign change for the right hand side forcing term changes the problem drastically. Firstly, the problem becomes unstable and solutions are no longer unique. Secondly, quasi-convexity which is a key tool in the study of obstacle problems, is unavailable because of the sign change. Also, in the parabolic setting, our approach to studying the free boundary of \eqref{e:main} is different than in \cite{mw07} for the elliptic problem \eqref{e:elliptic}. In \cite{mw07}, the authors use the implicit function theorem (where the gradient is nonvanishing) to immediately obtain $C^{1,\alpha}$ regularity of the free boundary. Standard techniques then improve $C^{1,\alpha}$ regularity of the free boundary to real analyticity. The difficulties in \cite{mw07} were to prove that the gradient does not vanish for energy minimizing solutions of \eqref{e:elliptic}, and also bound the size of the set where the gradient vanishes for all solutions of \eqref{e:elliptic}. For the free boundary problem in the parabolic setting \eqref{e:main}, the implicit function theorem is unavailable since $\partial_t u$ is not continuous. In fact, in Section \ref{s:unbound} we construct an example where $u_t$ is unbounded. Therefore, our approach to the free boundary in the parabolic setting is necessarily different from the elliptic setting.  
  
  We consider the situation in which heat advances, and so we will assume that $u_t \geq c$. Our main result is the following. 

  \begin{theorem} \label{t:main}
    Let $u$ be a solution to \eqref{e:main} in $Q_1$ and assume that $\partial_t u \geq c$ on 
    $\partial_p Q_1$. Then for any $0<\alpha<1$, the free boundary $\partial\{u>0\}$ is locally a $C^{1,\alpha/2}$ function of the spatial variable $x$. 
  \end{theorem}

   Our approach to proving Theorem \ref{t:main} is the following. We utilize known regularity estimates of solutions to \eqref{e:main} as well as the property $u_t \geq c$ to obtain that the free boundary is a Lipschitz graph over the spatial variable $x$. We then consider points $(x_0,t_0)$ where the spatial gradient is nonvanishing, i.e. $|\nabla u(x_0,t_0)|>0$. Of course, the implicit function theorem will give local $C^{1,\alpha}$ estimates in a neighborhood of $(x_0,t_0)$ of the spatial free boundary $\partial\{u(\cdot, t_0)>0\}$. However, since it is not known a priori that $u_t$ is continuous in a neighborhood of $(x_0,t_0)$ we cannot invoke the implicit function theorem to obtain $C^{1,\alpha}$ regularity of the space-time free boundary $\Gamma$. Instead, we utilize the hodograph transform in the spatial variable and obtain a nonlinear parabolic equation which will show that $u_t$ is H\"older continuous in a neighborhood of $(x_0,t_0)$. We then prove that the free boundary $\Gamma$ is differentiable with derivative zero at points where the spatial gradient vanishes, and thus the $\Gamma$ is differentiable at all points. As shown in our constructed example, $\partial_t u$ is not necessarily bounded. Consequently, we prove a bound on how the H\"older continuity of $\partial_t u$ diverges while the spatial gradient decays. Using this bound, we are able to obtain the $C^{1,\alpha/2}$ regularity of the full space-time free boundary. 

   %In free boundary problems, the free boundary typically possesses higher regularity than the solution itself. For instance, in this situation $u_t$ is not bounded and yet the free boundary is $C^{1,\alpha/2}$. Also typical of free boundary problems, is that higher regularity of the free boundary follows $C^{1,\alpha}$ regularity of the free boundary unless obvious restrictions (such as nonsmooth variable coefficients) restrict higher regularity. This is the case, for instance, in the elliptic problem \eqref{e:elliptic} in \cite{mw07}. Once the authors obtain $C^{1,\alpha}$ regularity of the free boundary, then real analatyicity of the free boundary follows. However, for the parabolic problem the free boundary for \eqref{e:main} is a novel example in which the free boundary is $C^{1,\alpha/2}$, the coefficients are smooth (even constant), and yet the free boundary need not be even $C^{2,\alpha}$. As a consequence of the example we construct in which $u_t$ is not bounded, we obtain as a consequence that the free boundary cannot be $C^{2,\alpha}$. 

 \subsection{Outline of the paper}
  In Section \ref{s:examples} we provide several examples of solutions to \eqref{e:main} that illustrate possible behavior such as non-uniqueness. In Section \ref{s:er} we state known regularity results for solutions to \eqref{e:main} and show solutions exist. In Section \ref{s:regularity} we prove our main result Theorem \ref{t:main}. In Section \ref{s:unbound} we show that a solution to \eqref{e:main} can have an unbounded time derivative. 
   
 \subsection{Notation}
  \begin{itemize}
      \item Throughout the paper $(x,t) \in \mathbb{R}^n \times (-\infty,\infty)$. The distance in Euclidean space $\mathbb{R}^n$ will be denoted by $|x-y|=\sqrt{\sum_{i=1}^n (x_i-y_i)^2}$. As a special case, the distance in $\mathbb{R}$ for $|t-s|=\sqrt{(t-s)^2}$ will have the same notation. 
      \item The spatial gradient is denoted by $\nabla u = (u_{x_1}, \ldots, u_{x_n})$. 
      \item The Hessian of $u$ in the spatial variable $x$ is denoted by $D^2 u$. 
      \item $Q_r(x,t):=\{(y,s): |x-y|+|t-s|^{1/2} < r\}$. We note that this definition differs from some in the literature which requires $s\leq t$. We find it more convenient to allow $s>t$.  
      \item When $(x,t)=(0,0)$, we will simply denote $Q_r(0,0)=Q_r$. 
      \item $\partial_{p} Q_r(x,t):= (B_r \times \{t-r^2\})\cup(\partial B_r \times [t-r^2,t+r^2))$, the parabolic boundary. 
      \item $u \in C^{\alpha}(Q_r(x,t))$ the H\"older space, with norm 
      \[
       \| u \|_{C^{\alpha}(Q_r(x,t))}:=\sup_{Q_r(x,t)} |u| + \sup_{\substack{(y,s),(z,\tau) \in Q_r(x,t)\\(y,s) \neq (z,\tau) }} \frac{|u(y,s)-u(z,\tau)|}{(|x-y|^2 + |t-s|^2)^{\alpha/2}}. 
      \]
      \item Due to the natural parabolic scaling, we also consider the H\"older space $C^{\alpha,\alpha/2}(Q_r(x,t))$ with 
      \[
      \| u \|_{C^{\alpha,\alpha/2}(Q_r(x,t))}:= \sup_{Q_r(x,t)} |u| + \sup_{\substack{(y,s),(z,\tau) \in Q_r(x,t) \\ (y,s) \neq (z,\tau)}} \frac{|u(y,s)-u(z,\tau)|}{(|x-y|^\alpha + |t-s|^{\alpha/2})}
      \]
      We note that from \eqref{e:alphabound} the above norm is topologically equivalent to others found in the literature. 
      \item $\Gamma := \partial \{u(x,t)>0\}$ is the space-time free boundary. 
  \end{itemize}

\section{Instructive Examples} \label{s:examples}
  Here we provide some instructive examples to understand what one can and cannot expect from solutions to \eqref{e:main}. 
  \begin{example} \label{ex:timeonly}
   Let $u(x,t):=\max\{t,0\}$. 
  \end{example}
  Example \ref{ex:timeonly} shows that the time derivative $u_t$ need not be continuous. Also, 
  example \ref{ex:timeonly} illustrates the nonuniqueness phenomenon as any translation $v(x,t):=u(x,t-\tau)$ with $\tau>0$ has the same initial data time $t=0$. A local version illustrating the same principles is given as follows.  
  \begin{example} \label{ex:local}
   Let $u$ be the unique solution to 
   \[
   \begin{cases}
       u_t - \Delta u=1 \quad \text{ in } B_1 \times (0,1) \\
       u=0 \quad \text{ on } (\partial B_1 \times (0,1))\cup(B_1 \times \{0\}). 
   \end{cases}
   \] 
  \end{example}
  Note that $v\equiv 0$ will also be a solution to \eqref{e:main} with the same boundary data as well as any translation in time $v(x,t):=u(x,t-\tau)$ with $\tau>0$. Notice again that for the translations $v_t$ is not continuous. This example also illustrates that $v_t$ will not satisfy a strong minimum principle. More specifically, notice that $v_t \geq 0$ and $v_t$ is not identically zero; however, $v_t$ takes zero values on the interior.

  We now state a nontrivial example of global nonuniqueness which was constructed in \cite{gh92}. 
  \begin{example} \label{ex:selfsimilar}
     There exists a solution to 
     \[
     \begin{cases}
      u_t - u_{xx}=\chi_{\{u>0\}} \quad \text{ in } \mathbb{R} \times (0,\infty) \\
      u(x,0)=-c_1x^2 \quad \text{ on } \mathbb{R} \times \{0\}, 
      \end{cases}
     \]
     which is self-similar, i.e. satisfying $u(rx,r^2t)=r^2u(x,t)$, and $u(0,1)>0$.  
  \end{example}
  Note that $v(x,t)=-c(x^2+2t)$ is a different solution with the same initial data. Analogous self similar solutions to Example \ref{ex:selfsimilar} in higher dimension were constructed recently in \cite{fg24}. 

  Any solution to the elliptic problem \eqref{e:elliptic} will also be a solution to \eqref{e:main}, so the example constructed in \cite{aw06} with a cross-singularity is also an example for \eqref{e:main}. 
  \begin{example} \label{ex:elliptic}
      There exists a time-independent solution to \eqref{e:main} with a cross-singularity at $(0,t)$ for all $t$.  
  \end{example}
  Example \ref{ex:elliptic} illustrates that it is possible for the spatial gradient to vanish as well as illustrating that the spatial Hessian $|D^2 u|$ need not be bounded. Example \ref{ex:elliptic} also shows that without the assumption $u_t \geq c$, one cannot expect differentiability of the free boundary $\Gamma$. Since Examples \ref{ex:elliptic} and \ref{ex:local} do not satisfy our assumption $u_t \geq c$, we provide an additional example with $u_t \geq c$ which will also show that we cannot expect $u_t$ to be bounded. 

  The following example will be the most important for this paper and will illustrate that $u_t$ need not be bounded which we prove in Section \ref{s:unbound}. 
  \begin{example} \label{ex:last}
      Let $u$ be the least solution (constructed as shown in Section \ref{s:er}) to 
      \[
     \begin{cases}
      u_t - \Delta u=\chi_{\{u>0\}} \quad \text{ in } B_1 \times (0,\infty) \\
      u(x,0)=2x^2-1 \quad \text{ on } B_1 \times \{0\} \\
      u(x,t)=t+1 \quad \text{ on } \partial B_1 \times (0,\infty).
      \end{cases}
     \]
  \end{example}
  As time increases in Example \ref{ex:last}, the set $\{u(\cdot,t)<0\}$ will be an open interval which will shorten in length and collapse to a point. In dimension $n=1$, neither $|u_t|$ nor $|D^2 u|$ will be bounded on $\Gamma$.

\section{Regularity and Existence} \label{s:er}
In this section we prove existence and regularity estimates for solutions to \eqref{e:main}. Since the right hand side of \eqref{e:main} satisfies 
$0 \leq \chi_{\{u>0\}} \leq 1$, we first state some a priori estimates to solutions of the linear equation 
\begin{equation} \label{e:w}
 \begin{cases}
   &u_t - \Delta u = f \quad \text{ in } Q_1 \\
   &u=\psi \quad \text{ on } \partial_p Q_1.
 \end{cases}
 \end{equation}
 We assume that $f\in L^{\infty}$ and state some bounds for $u$ and its derivatives. First, $u,\nabla u, D^2 u, u_t \in L^p(Q_1)$ for any $1\leq p<\infty$. More precisely, see Theorem 7.13 in \cite{l96}, we have the bounds
 \begin{equation} \label{e:lp}
  \int_{Q_{1/2}} |u_t|^p + |\nabla u|^p+ |D^2 u|^p \leq C(p) \int_{Q_1} |f|^p + |u|^p.
 \end{equation}
 If the boundary data is smooth, then we have the bounds up to the boundary:
\[
 \int_{Q_{1}} |u|^P + |u_t|^p + |\nabla u|^p+ |D^2 u|^p 
 \leq \int_{Q_{1}}|f|^p+ |\psi|^p + |\psi_t|^p + |\nabla \psi |^p+ |D^2 \psi|^p.  
\]
From the Sobolev embedding theorem for dimension $n+1$, the bound \eqref{e:lp} implies that 
\begin{equation} \label{e:holderu}
 \| u \|_{C^{\alpha}(Q_{1/2})}
 \leq C(\alpha) ( \sup_{Q_1} |u| +\| f \|_{L^{\infty}}) \quad \text{ for any } 0< \alpha < 1. 
\end{equation}
Since \eqref{e:holderu} shows that $u$ is continuous, then $\{u>0\}$ is open, and the free boundary $\Gamma$ is well-defined. 

Since $f \in L^{\infty}$, we also have H\"older estimates on the spatial derivatives of $w$. From Theorem 4.1 on page 584 of \cite{l68}, there exists $\gamma>0$ and a constant $C$ such that a solution $w$ satisfies
\begin{equation} \label{e:holdergrad}
 \|\nabla u\|_{C^{\gamma}(Q_{1/2})} \leq C (\sup_{Q_1} |u| + \| f\|_{L^{\infty}}). 
\end{equation}
Again, if $\psi$ is smooth, then we have estimates up to the boundary. We note that in the above estimate, one may either consider H\"older regularity for the  scalar function $|\nabla u|$, or H\"older regularity for each of the components separately for the vector-valued function $\nabla u$. 

Using a scaling argument, we can upgrade the H\"older continuity of the gradient. 
\begin{lemma} \label{l:holderspace}
 Let $u$ be a solution to \eqref{e:w}.  For any $\alpha<1$, there exists a constant $C$ depending on $\alpha$ such that 
 \[
 \| \nabla u \|_{C^{\alpha,\alpha/2}(Q_{1/2})} \leq C (\sup_{Q_1}|u| +\|f\|_{L^{\infty}}). 
 \]
\end{lemma}

\begin{proof} 
By the linearity of the equation, we can divide out by $\displaystyle \sup_{Q_1}|u| + \| f \|_{L^{\infty}}$ and thus assume $\displaystyle\sup_{Q_1}|u| + \| f \|_{L^{\infty}}\leq 1$.
 From \eqref{e:holdergrad}, we have the gradient is universally bounded and continuous. 
  For a solution $u$ we define
 \[
 S_r(u,(x,t)):= \sup_{Q_r(x,t)} |u(y,s)-u(x,t)-\nabla u(x,t) \cdot (y-x)|
 \]
  We must show for a point $(x,t) \in Q_{1/2}$ that there exists a constant $C$ such that 
 \[
  S_r(u,(x,t))\leq Cr^{1+\alpha}.   
 \]
 We employ a scaling argument to obtain a proof by contradiction, see for instance Theorem 6.1 in \cite{alp15}. Suppose by way of contradiction that no constant $C$ exists for the above inequality Then there exists a sequence of solutions $u_k$ with right hand sides $f_k$ and points $(x_k,t_k)$ such that 
 \[
 \frac{S_{r_k}(u_k, (x_k,t_k))}{r_k^{1+\alpha}} \to \infty,
 \]
 and 
 \begin{equation} \label{e:liouville}
 S_{r_k 2^j}(u_k,(x_k,t_k))\leq 2^{j(1+\alpha)}S_{r_k}(u_k,(x_k,t_k)) \quad \text{ for all integers } j\leq k. 
 \end{equation}
 We rescale by 
 \[
 u_{r_k}(y,s):= \frac{u(r_ky+x,r_k^2s+t)-u(x,t)-\nabla u(x,t)\cdot(y-x)}{S_{r_k}(u_k,(x_k,t_k))}.
 \]
 Notice that 
 \begin{equation} \label{e:caloric}
 |(\partial_t - \Delta) u_{r_k}|\leq \frac{r_k^2}{S_{r_k}(u_k,(x_k,t_k))}\leq r_k^{1-\alpha} \quad \text{ on } Q_{r^{-k}}(0,0). 
 \end{equation}
 Then from \eqref{e:lp}, \eqref{e:holderu} and \eqref{e:holdergrad} we have that $u_{r_k}$ converges to a limiting solution $u_0$ with the following convergence 
 \[
 u_k \to u_0, \ \nabla u_k \to \nabla u_0 \text{ in } C^{\gamma}, \ \text{ and }  \partial_t u_k \to \partial_t u_0, \  D^2 u_k \rightharpoonup D^2 u_0 \ \text{ in } L^p
 \]
 on all compact subsets of $\mathbb{R}^{n} \times (-\infty,\infty)$. Then we have the following properties for $w_0$:
 \[
 \begin{aligned}
  &(1)\quad (\partial_t-\Delta)u_0 = 0 \quad \text{ in } R^n \times (-\infty,\infty)
  &\text{ from } \eqref{e:caloric}\\
  &(2) \quad S_R(u_0,(0,0))\leq CR^{1+\alpha} \text{ for } R\geq 1 &\text{ from } \eqref{e:liouville} \\
  &(3) \quad u_0(0,0)=0, \ |\nabla u_0(0,0)|=0 &\text{ from $\gamma$-H\"older convergence} \\
  &(4) \quad S_1(u_0,(0,0))=1 &\text{ from $\gamma$-H\"older convergence}.
 \end{aligned}
 \]
 From properties $(1)$ and $(2)$ above, we can utilize the following standard bound for caloric functions (see \cite{e10})
 \[
 \sup_{Q_{R}} |u_t| \leq \frac{C}{R^{n+2+2}} \| u \|_{L^1(Q_{2R})}\leq \frac{C}{R^{n+2+2}} R^{n+2+1+\alpha}\leq C R^{\alpha-1}. 
 \]
 Letting $R \to \infty$ we conclude $|u_t|\equiv 0$. A similar bound and argument proves that $|D^2 u|\equiv 0$. Then $u_0$ is an affine function. From the property $(3)$ above $u_0 \equiv 0$. However, this contradicts the property $(4)$ listed above. 
\end{proof}

 We will often implicitly utilize the following elementary estimate for $x,y\geq0$ when dealing with H\"older regularity
 \begin{equation} \label{e:alphabound}
  (x+y)^{\alpha} \leq x^{\alpha}+y^{\alpha} \leq 2(x+y)^{\alpha}. 
 \end{equation}
 We also note that since $u \in C^{\alpha,\alpha/2}$ for every $\alpha<1$, then if $0< \alpha<\beta<1$,
 then 
 \[
 \begin{aligned}
 |u(x,t)-u(y,s)| &\leq C(\beta)(|x-y|+ |t-s|^{1/2})^{\beta}\\
 &= C(\beta) (|x-y|+ |t-s|^{1/2})^{\beta-\alpha}(|x-y|+ |t-s|^{1/2})^{\alpha}. 
 \end{aligned}
 \]
 Thus, if $C(\beta) (|x-y|+ |t-s|^{1/2})^{\beta-\alpha}\leq 1$, then 
 \begin{equation} \label{e:1holder1}
 |u(x,t)-u(y,s)| \leq (|x-y|+ |t-s|^{1/2})^{\alpha}. 
 \end{equation}
 Throughout the paper, for convenience we will often assume $(x,t)$ and $(y,s)$ are close enough so that \eqref{e:1holder1} holds, and we will therefore omit the constant $C$.

With the appropriate a priori estimates, we now prove existence of solutions to  \eqref{e:main}. We will construct a least solution. We consider a right hand side defined by 
\begin{equation} \label{e:rhs}
 f_{\epsilon}(x):=
 \begin{cases}
  0 &\text{ if } x<0 \\
  x/\epsilon  &\text{ if } 0\leq x \leq \epsilon \\
  1  &\text{ if } x>0. 
 \end{cases}
\end{equation}

The following existence theorem is classical since the right hand side $f$ is Lipschitz continuous. For instance, one could use the method of sub and supersolutions (shown in the elliptic case in section 9.3 in \cite{e10}), or a fixed point theorem (shown in the parabolic case with zero lateral boundary in Chapter 10 in \cite{r13}).
\begin{lemma} \label{l:eandu}
 Let $\Omega\subset \mathbb{R}^n$ be a smooth bounded domain, and let $\psi(x,t)$ be a smooth function. There exists a unique classical solution $u^{\epsilon}$ to 
 \begin{equation} \label{e:epsilon}
 \begin{cases}
  u^{\epsilon}_t - \Delta u^{\epsilon} = f_{\epsilon}(u^{\epsilon}) \quad \text{ in } \Omega \times (0,T) \\
  u^{\epsilon}=\psi \quad \text{ on } (\Omega \times \{0\})\cup (\partial \Omega \times (0,T))
 \end{cases}
 \end{equation}
\end{lemma}

%\begin{proof}
%    We will use the method of sub- and supersolutions. We let $u_0$ be the %unique solution to 
%    \[
% \begin{cases}
%\partial_t u_0 - \Delta u_0 = 0 \quad \text{ in } \Omega \times (0,T) \\
%  u_0=\psi \quad \text{ on } (\Omega \times \{0\})\cup (\partial \Omega \times 
%(0,T)). 
% \end{cases}
% \]
% We then recursively define $u_{k+1}$ to be the unique solution to 
% \[
% \begin{cases}
%\partial_t u_{k+1} - \Delta u_{k+1} = f_{\epsilon}(u_k) \quad \text{ in } \Omega %\times (0,T) \\
%  u_{k+1}=\psi \quad \text{ on } (\Omega \times \{0\})\cup (\partial \Omega %\times (0,T)). 
% \end{cases}
% \]
% From the maximum principle, $u_1 \geq u_0$. We now assume that $u_{k} \geq %u_{k+1}$. Then 
% \[
% \begin{cases}
%\partial_t (u_{k+1}-u_k) - \Delta (u_{k+1}-u_k) = f_{\epsilon}(u_k)-f_{\epsilon}
%(u_{k-1}) \quad \text{ in } \Omega \times (0,T) \\
%  u_{k+1}-u_k=0 \quad \text{ on } (\Omega \times \{0\})\cup (\partial \Omega %\times (0,T)). 
% \end{cases}
% \]
% Since $f_{\epsilon}$ is nondecreasing, we obtain from the maximum principle %that $u_{k+1} \geq u_{k}$. Now there is a unique solution $v$ to 
% \[
% \begin{cases}
%\partial_t v - \Delta v = 1 \quad \text{ in } \Omega \times (0,T) \\
%  v=\psi \quad \text{ on } (\Omega \times \{0\})\cup (\partial \Omega \times 
%(0,T)). 
% \end{cases}
% \]
% Since $f$ is Lipschitz and nondecreasing it follows that a unique solution %exists. 
%\end{proof}

We now prove the existence of solutions to \eqref{e:main}. 
\begin{theorem} \label{t:e}
 Let $\Omega\subset \mathbb{R}^n$ be a smooth bounded domain, and let $\psi$ be smooth. There exists a solution $u$ to 
 \[
 \begin{cases}
  u_t - \Delta u = \chi_{\{u>0\}} \quad \text{ in } \Omega \times (0,T) \\
  u=\psi \quad \text{ on } (\Omega \times \{0\})\cup (\partial \Omega \times (0,T))
 \end{cases}
 \]
\end{theorem}

\begin{proof}
 Let $u_{\epsilon}$ be the solution to \eqref{e:epsilon}. Letting $\epsilon \to 0$, there exists a limiting function $u$ with 
 \[
 \begin{aligned}
  &u^{\epsilon} \to u_0 \quad \text{ in } C^{\alpha}( \Omega \times (0,T))\\
  &\nabla u^{\epsilon} \to \nabla u_0 \quad \text{ in } C^{\alpha,\alpha/2}( \Omega \times (0,T))\\
  &u_t^{\epsilon}, \ D^2 u^{\epsilon} \rightharpoonup \partial_t u_0, \ D^2 u_0 \quad \text{ in } L^p( \Omega \times (0,T)). 
 \end{aligned}
 \]
 If $u(x,t)>0$, then $u^{\epsilon}>0$ in $Q_r(x,t)$ for some $r>0$, and it follows that $\lim_{\epsilon \to 0} f_{\epsilon}(u^{\epsilon}(x,t))=f(u(x,t))$. 
 Since $u^{\epsilon_2}$ is a subsolution with right hand side $f_{\epsilon_1}$, it follows that   $u^{\epsilon_2} \leq u^{\epsilon_1}$ for $\epsilon_1 \leq \epsilon_2$ and that $u^{\epsilon} \nearrow u$. If $u(x,t)\leq 0$, then $u^{\epsilon}(x,t) \leq 0$ for all $\epsilon>0$. So, for all points, 
\[
\lim_{\epsilon \to 0} f^{\epsilon}(u_{\epsilon}) = f(u) 
\]
If $\phi \in C_0^{\infty}(\Omega \times (0,T))$, then from weak convergence in $L^P$, we have 
 \[
\lim_{\epsilon \to 0} \int_{\Omega}\int_0^T (\partial_t u^{\epsilon}- \Delta u^{\epsilon}) \phi = \int_{\Omega}\int_0^T (\partial_t u- \Delta u) \phi. 
 \]
 Then 
 \[
 \begin{aligned}
 f(u)&= \lim_{\epsilon \to 0} f_{\epsilon}(u_{\epsilon})\\
 &= \lim_{\epsilon \to 0} \int_{\Omega}\int_0^T (\partial_t u_{\epsilon}- \Delta u_{\epsilon}) \phi \\
 &= \int_{\Omega}\int_0^T (\partial_t u- \Delta u) \phi
 \end{aligned}
 \]
 Since $\phi$ is arbitrary it follows that $u$ is a strong solution and satisfies the equation for almost every $(x,t) \in \Omega \times (0,T)$. 
\end{proof}

We now show that for the least solutions constructed above that the solution inherits the interior bound $u_t \geq c$ from the boundary. This result shows that solutions with $u_t \geq c$ exist. 
\begin{proposition} \label{p:increase}
 Assume that $\partial_t \psi \geq c$. Then $u(x,t_2)-u(x,t_1)\geq c (t_2-t_1)$ and $\partial_t u(x,t)\geq c$ whenever $\partial_t u$ exists. 
\end{proposition}

\begin{proof}
 Let $u_{\epsilon}$ be the solution to \eqref{e:epsilon}. Then 
 \[
 \begin{cases}
  \partial_t (\partial_t u_{\epsilon}) - \Delta (\partial_t u_{\epsilon}) = f_{\epsilon}'(u_{\epsilon}) (\partial_t u_{\epsilon})= f_{\epsilon}'(u_{\epsilon}) (\partial_t u_{\epsilon})^+ \quad \text{ in } \Omega \times (0,T) \\
  \partial_t u_{\epsilon}=\psi_t \quad \text{ on } (\Omega \times \{0\})\cup (\partial \Omega \times (0,T)). 
 \end{cases}
 \]
 Since $\psi_t \geq c>0$ and since $f_{\epsilon}'(u_{\epsilon}) (\partial_t u_{\epsilon})^+ \geq 0$ it follows from the minimum principle that $\partial_t u_{\epsilon} \geq c$. Then $u_{\epsilon}(x,t_2)-u_{\epsilon}(x,t_1)\geq c(t_2-t_1)$. From the H\"older convergence, it follows that $u(x,t_2)-u(x,t_1)\geq c(t_2-t_1)$. Then also $\partial_t u \geq c$ whenever the derivative exists.  
\end{proof}

\section{Free Boundary Regularity} \label{s:regularity}
 In this section we assume that we have a solution to \eqref{e:main}, and that 
 \begin{equation} \label{e:lip}
 u(x,t_2)-u(x,t_1) \geq c(t_2 - t_1).
 \end{equation}
 Proposition \ref{p:increase} shows that \eqref{e:lip} will be true for the least solutions as long as the boundary data satisfies the same condition. This assumption is natural for the applications in combustion when assuming that the heat is advancing. 

 \begin{lemma} \label{l:lipschitz}
   Assume that $u$ is a solution to \eqref{e:main} in $Q_1$ and satisfies \eqref{e:lip}. Then the free boundary $\Gamma$ is a Lipschitz graph over the spatial variable $x$.
 \end{lemma}

 \begin{proof}
     From \eqref{e:lip}, it is clear that $\Gamma$ is a graph over the spatial variable $x$, that is for $(x,t) \in \Gamma\cap Q_1$, there exists a function $H(x)$ such that $(x,t)=(x,H(x))$. 
     Recall from Lemma \ref{l:holderspace} the $L^{\infty}$ estimate
     \[
      \sup_{Q_{r}} |\nabla u(x,t)|\leq C_r \quad \text{ for any } r<1. 
     \]
    If $(x,t) \in \Gamma$, then 
    \[
     u(y,t)\geq -C_r |x-y|. 
    \]
    Also, if $s\geq t$, then 
    \[
    u(y,s)\geq c(s-t)- C_r|x-y|. 
    \]
    Thus, if $s-t\geq \frac{C_r}{c}|x-y|$, then $u(y,s)\geq 0$, so that 
    \[
      H(y)-H(x)\leq \frac{C_r}{c}|x-y|
    \]
    A similar estimate holds from below, so that we conclude 
    \[
    -\frac{C_r}{c}|x-y| \leq H(y)-H(x)\leq \frac{C_r}{c}|x-y|,
    \]
    and we have shown that $\Gamma$ is a Lipschitz graph over the spatial variable $x$ with a Lipschitz bound in the interior of $Q_1$.  
 \end{proof}
 
 We will prove that the space-time free boundary is $C^{1,\alpha}$. This will be accomplished in three steps. The first step is to show that locally near a point where $|\nabla u(x,t)|\neq 0$, the space-time free boundary is $C^{\infty}$. This will be done in Section \ref{s:pos}. The second step consists of showing that the space-time free boundary is differentiable at points where $|\nabla u(x,t)|= 0$. The third step is to piece together the estimates to show that the space-time free boundary is $C^{1,\alpha}$. The second and third steps will be accomplished in Section \ref{s:allholder}.

 \subsection{Regularity of $\Gamma$ when $|\nabla u(x,t)|>0$.} \label{s:pos}
 We first consider a point $(x,t) \in \Gamma$ at which $|\nabla u(x,t)|> 0$. By rotation we may assume that $\nabla u(x,t)=\partial_{x_n} u(x,t)$. We now rescale the function $u$. Fix $M>0$ large, and let $r$ satisfy
 \begin{equation} \label{e:r}
 r^{\alpha}=\frac{|\nabla u(x,t)|}{M}.
 \end{equation}
 We also define the rescaling
 \begin{equation} \label{e:ur}
 u_r(y,s):=\frac{u(ry+x,r^2s+t)}{|\nabla u(x,t)|r}=\frac{u(ry+x,r^2s+t)}{Mr^{1+\alpha}}. 
 \end{equation}
 We list some elementary properties for $u_r$ on $Q_1$.

 \begin{proposition} \label{p:rescale}
  Assume that $u$ is a solution to \eqref{e:main} in $Q_{R}(x,t)$ and satisfies \eqref{e:lip}. Fix $0<\alpha<1$, and let $r$ and $u_r$ be defined as in \eqref{e:r} and \eqref{e:ur} respectively. If $2r < R$, then after possible rotation $u_r$ satisfies the following properties in $Q_1$. 
  \begin{align}
     &\partial_n u_r(0,0)= \nabla u_r(0,0)=(0,\ldots,1) \label{e:q1}\\
     &1-1/M \leq |\nabla u(y,s)|\leq 1+1/M \label{e:q2}\\
     & |\partial_i u_r(y,s)|\leq 2/M \label{e:q3} \quad \text{ for } 1\leq i <n \\
     & |u_r(y,s)|\leq 1+2/M \label{e:q4} \\
     & (\partial_t - \Delta)u_r = r^{1-\alpha} \chi_{\{u_r>0\}}\label{e:q5}.
 \end{align}
 \end{proposition}
 
 \begin{proof}
  Properties \eqref{e:q1}, \eqref{e:q2}, and \eqref{e:q3} come from the standard scaling and the fact that $\| \nabla u \|_{C^{\alpha,\alpha/2}}(Q_1) \leq 1$. Since also 
  $\| u \|_{C^{\alpha}(Q_1)} \leq 1$ it follows that 
  \[
  |u_r(0,s)|=|u_r(0,s)-u_r(0,0)|=\frac{|u(x,r^2s+t) -u(x,t)|}{Mr^{1+\alpha}}
  \leq \frac{1+\alpha}{M(1+\alpha)} \leq M^{-1}. 
  \]
  Combined with the gradient estimate \eqref{e:q2} we obtain \eqref{e:q4}. 
  Finally, from rescaling we obtain \eqref{e:q5}. 
  %\[
  %(\partial_t-\Delta) u_r = r^{1-\alpha} \chi_{\{u_r > 0\}}. 
  %\]
 \end{proof}

   For $|\nabla u(x,t)|>0$ we rotate and rescale with $u_r$ to obtain the properties in Proposition \ref{p:rescale}. For notational convenience, throughout the remainder of this subsection we will write $u_r$ as $u$ in the following discussion. Thus, we assume that $u$ satisfies all the properties in Proposition \ref{p:rescale}. We perform the Hodograph transform by letting $v(x',x_n,t)$ be the unique function defined by 
 \begin{equation} \label{e:hodograph}
     u(x',v(x',x_n,t),t)=x_n. 
 \end{equation}
 By the inverse function theorem, we have that $v$ exists and has the same regularity as $u$. The derivatives of $v$ can be computed as follows 
 \[
 \begin{aligned}
     &u_n v_n = 1 \\
     &u_i + u_n v_i =0 \\
     &u_t + u_n v_t =0 \\
     &u_{nn} v_n^2 + u_n v_{nn} = 0\\
     &u_{in} v_n + u_{nn} v_i v_n + u_n v_{in}=0 \\
     &u_{ii}+2u_{ni}v_i + u_{nn} v_i^2 + u_n v_{ii}=0. 
 \end{aligned}
 \]
 As shown in \cite{aks25}, the term $\Delta u$ can be rewritten as a divergence form elliptic operator for $v$. Thus, $v$ is a strong solution to the equation 
 \[
  -\frac{v_t}{v_n}- \left[\left(1+ \sum_i v_i^2 \right)\frac{1}{2v_n^2} \right]_n
  +\sum_i \left[\frac{v_i}{v_n} \right]_i= r^{1-\alpha} \chi_{\{x_n >0\}}. 
 \]
To illustrate the method, we formally differentiate the above equation in $t$ to obtain the following equation for $v_t$ 
\[
\begin{aligned}
&-\frac{v_n (v_{t})_t-v_t (v_t)_n}{v_n^2}
-\left[\sum_i \frac{v_i(v_t)_i}{v_n^2}- \left(1+ \sum_i v_i^2 \right)\frac{(v_{t})_n}{v_n^3} \right]_n \\
&\quad + \sum_i \left[\frac{v_n (v_t)_i - v_i (v_t)_n}{v_n^2} \right]_i =0. 
\end{aligned}
\]
Then formally, $w=v_t$ solves the following divergence-form parabolic equation 
\begin{equation} \label{e:formal}
\int_U \frac{1}{v_n} w_t \psi + A^{ij}(\nabla v) w_i \psi_j -\frac{v_t}{v_n^2}w_n \psi =0
\end{equation}
for almost every $t$ and for $\psi \in W_0^{1,2}(U)$. The matrix $A^{ij}(\nabla v)$ is given specifically as 
\[
	A^{ij}(\nabla v) =  \begin{bmatrix}
		\frac{1}{v_n} & 0 & \cdots & 0 & -\frac{v_1}{v_n^2} \\
		0 & \frac{1}{v_n} & \cdots & 0 & -\frac{v_2}{v_n^2} \\
		\vdots & \vdots & &\vdots & \vdots \\
		0 & 0 & \cdots & \frac{1}{v_n} &-\frac{v_{n-1}}{v_n^2} \\
		-\frac{v_1}{v_n^2} & -\frac{v_2}{v_n^2} & \cdots & -\frac{v_{n-1}}{v_n^2} & \frac{1}{v_n^3}(1 + \sum_i v_i^2)
	\end{bmatrix}
\]
From Proposition \ref{p:rescale} and the relations for the derivatives of $u$ and $v$, the above matrix will be uniformly elliptic on $Q_1$. By Nash's theorem (see Theorem 6.28 in \cite{l96}) we have that $ v_t= w \in C^{0,\gamma}$ for some $\gamma>0$. But then the coefficients of \eqref{e:formal} are H\"older continuous. Then by Theorem 6.45 in \cite{l96}, the spatial derivatives of $w$ are H\"older continuous and $w=v_t \in C^{(1+\gamma)/2}$, so by bootstrapping we obtain $w=v_t \in C^{(1+\alpha)/2}$ for any $0<\alpha<1$. 

\begin{remark}
    The above arguments can be iterated for both the time direction $t$ as well as $i<n$ (see for instance \cite{ks24}) to conclude that $\Gamma \in C^{\infty}$ in a neighborhood of a point $(x,t)$ where $|\nabla u(x,t)|>0$. Since we do not need this higher regularity in this paper, and since $\Gamma$ is not necessarily $C^{\infty}$ everywhere, we do not provide the details.  
\end{remark}

 The above computations were done formally. Consequently, we instead consider the equation solved by the difference quotient $w_h(x,t):= (v(x,t+h)-v(x,t))/h$. 

\begin{theorem} \label{t:hodograph}
 Assume that $u$ satisfies the properties in Proposition \ref{p:rescale}, and let $v$ be the Hodograph transform in the $x_n$ variable. For any $0<\alpha<1$, there exists a constant $C_1$ and an $r_1>0$ such that the difference quotient $w_h(x,t):= (v(x,t+h)-v(x,t))/h \in C^{(1+\alpha)/2}$ and 
 \[
  \| w_h \|_{C^{(1+\alpha)/2}(Q_r)} \leq C_1 \| w_h \|_{L^p(Q_{2r})} \quad \text{ independently of } h. 
 \]
 Consequently, the same estimate holds for $w=v_t$, and thus there exists a constant $C_2$ and
 such that 
 \begin{equation} \label{e:utholder}
  \| u_t \|_{C^{\alpha}(Q_{1/2})} \leq C_2 \sup_{Q_{1}} |u_t|. 
 \end{equation}
\end{theorem}

\begin{proof}
Throughout the proof we will suppress the dependence on the variable $x$, and only write the dependence on time $t$. The quotient $w_h$ satisfies the equation
\[
\int_{B_1} \frac{1}{v_n(t+h)} [w_h]_t \psi 
+ A^{ij}_h(t) [w_h]_i \psi_j 
- \frac{v_t(t)}{v_n(t+h)v_n(t)} [w_h]_n \psi=0, 
\]
where $A^{ij}_h$ is symmetric and 
\[
 \begin{aligned}
 A_h^{ij}&= \frac{\delta_{ij}}{v_n(t+h)}  &\quad \text{ if } 1\leq i,j <n, \\
 A_h^{in}&=A_h^{ni}=  -\frac{v_i(t)}{v_n(t)v_n(t+h)} &\quad \text{ if } 1\leq i <n, \\
 A_h^{nn}&= \frac{v_n(t+h)+v_n(t)}{2v_n^2(t)v_n^2(t+h)}(1 + \sum_i v_i^2(t)).
 \end{aligned}
\]
%where 
%\[
%     A^{ij}_h(t) =  \begin{bmatrix}
%		\frac{1}{v_n(t+h)} & 0 & \cdots & 0 & -\frac{v_1(t)}{v_n(t)v_n(t+h)} \\
%		0 & \frac{1}{v_n(t+h)} & \cdots & 0 & -\frac{v_2(t)}{v_n(t)v_n(t+h)} \\
%		\vdots & \vdots & &\vdots & \vdots \\
%		0 & 0 & \cdots & \frac{1}{v_n(t+h)} &-\frac{v_{n-1}(t)}{v_n(t)v_n(t+h)} \\
%		-\frac{v_1}{v_n(t)v_n(t+h)} & -\frac{v_2}{v_n(t)v_n(t+h)} & \cdots & -%\frac{v_{n-1}}{v_n(t)v_n(t+h)} & \frac{v_n(t+h)+v_n(t)}{2v_n^2(t)v_n^2(t+h)}(1 + %\sum_i v_i^2(t))
%	\end{bmatrix}
%\]
As explained in the formal discussion above, the above matrix is uniformly elliptic. From \eqref{e:q4} we have that $\| u_t \|_{L^p(Q_{1-M^{-1}})} \leq C(p)$. From the relations between the derivatives of $u$ and $v$, we have that $v,v_t \in L^p$ for all $1\leq p < \infty$. It then follows that $\| w_h \|_{L^p} \leq C(p)$, see for instance Theorem 3 in Section 5.8.2 in \cite{e10}. As $w_h$ satisfies a linear equation, we obtain a priori interior Sobolev estimates for $w_h$. Also, from Nash's Theorem, we have that $w_h \in C^{0,\gamma}$ uniformly for some $\gamma>0$. We can then let $h\to 0$, and we obtain that indeed $w=v_t$ solves the original formal equation. Thus, from the discussion above for the formal equation, the H\"older estimates for $v$ transfer to $u$. A covering argument and compactness of $\overline{Q}_{1/2}$, then gives \eqref{e:utholder}.
\end{proof}

\subsection{H\"older regularity of all of $\Gamma$} \label{s:allholder}
We now show that $\Gamma$ is differentiable when the spatial gradient vanishes.
\begin{lemma} \label{l:zeroder}
 Assume that $(x,t) \in \Gamma$ and $|\nabla_x u(x,t)|=0$. Then $\Gamma$ is differentiable at $(x,t)$ and $|\nabla \Gamma(x,t)|=0$. 
\end{lemma}

\begin{proof}
    Let $(x,t) \in \Gamma$ and assume that $|\nabla u(x,t)=0|$. We revisit the proof of Lemma \ref{l:lipschitz}, but utilize that $|\nabla u(x,t)=0|$, and use the H\"older estimate on the spatial gradient rather than just the $L^{\infty}$ estimate on the gradient. 
    
    From the fundamental theorem of Calculus
    \[
    u(y,t)= \int_0^{|y-x|} \nabla u(x+\xi,t) \cdot \frac{y-x}{|y-x|} \ d \xi. 
    \]
    Recall from Lemma \ref{l:holderspace} the H\"older estimate 
     \[
    |\nabla u(x,t)- \nabla u(y,t)| \leq C(|x-y|^{\alpha}+|t-s|^{\alpha/2}). 
    \]
    Then 
    \[
     u(y,t) \geq -\int_0^{y-x}C\xi^{\alpha} d \xi= -\frac{C}{\alpha+1} |y-x|^{\alpha+1}. 
    \]
    Then if $s>t$, we have 
    \[
    u(y,s)\geq -\frac{C}{\alpha+1} |y-x|^{\alpha+1}+ c(s-t). 
    \]
    A similar estimate holds from above so that as in the proof of Lemma \ref{l:lipschitz}, we have
    \begin{equation} \label{e:zeroalpha}
    -\frac{C}{c(\alpha+1)} |y-x|^{\alpha+1}\leq H(y)-H(x) \leq \frac{C}{c(\alpha+1)} |y-x|^{\alpha+1}. 
    \end{equation}
\end{proof}

We now have established that every point of the free boundary $\Gamma$ is differentiable. 
If $(x,t) \in \Gamma$, then we label $\nu_{x,t}$ as the unit normal to $\Gamma$ at $(x,t)$. We note that from the ideas in Lemma \ref{l:lipschitz}, we have that the unit normal lies in the cone $s\geq |y||\nabla u(x,t)|/c$, or written differently, the angle $\theta$ between $(0,1)$ and $\nu_{x,t}$ satisfies 
\[
\sin \theta\leq \frac{|\nabla u(x,t)|}{\sqrt{|\nabla u(x,t)|^2 + c^2}}. 
\]

We will not be able to show that $\partial_t u$ is H\"older continuous up to a point where the gradient vanishes. Indeed, as will be shown in Section \ref{s:unbound} for Example \ref{ex:last}, we cannot expect the time derivative to be continuous or even bounded. Thus, the H\"older estimate will blow up. We have to utilize that the spatial gradient is vanishing and balance this with how $\partial_t u$ is increasing. The next result is a corollary of Theorem \ref{t:hodograph} and follows from rescaling. 

\begin{corollary} \label{c:rescale}
 Assume $u$ is a solution to \eqref{e:main} in $Q_1$. Fix $0<\gamma<1-\alpha$, and assume that $|\nabla u(x,t)|>0$. Let $\rho$ be defined as $\rho^{1-\gamma}=|\nabla u(x,t)|/M$. There exists a constant $C_1$ depending on $\gamma$ such that 
 \begin{equation} \label{e:rho}
 \sup_{Q_{\rho/2}(x,t)} |\partial_t u| \leq C_1 M\rho^{-\gamma},
 \end{equation}
 and a constant $C_2$ depending on $\alpha,\gamma$ such that  
 \begin{equation} \label{e:rho2}
  | u_t(y,s)-u_t(x,t)| \leq C_2 r^{-2\gamma-\alpha} (|x-y|^{\alpha}+|t-s|^{\alpha/2}) \quad \text{ for any } (y,s) \in Q_{r/2}. 
 \end{equation}
\end{corollary}

\begin{proof}
 Fix $0<\alpha<1$ and $0<\gamma<1-\alpha$. Let $\rho$ be defined as $\rho^{1-\gamma}=|\nabla u(x,t)|/M$, and rescale with 
 \[
 u_\rho(y,s)=\frac{u(\rho y+x,\rho^2s+t)}{M\rho^{1-\gamma}}.
 \]
 From Theorem \ref{t:hodograph} 
 we have 
 \[
 \sup_{Q_{1/2}}|\partial_t u_\rho| \leq C.
 \]
 From the relation $\partial_t u_\rho(y,s)=\rho^{\gamma}M^{-1} u_t(\rho y+x,\rho^2s+t)$, we have
 \[
 \rho^{\gamma}M^{-1} \sup_{Q_{\rho/2}} |\partial_t u| \leq C,
 \]
 or 
 \[
 \sup_{Q_{\rho/2}} |\partial_t u| \leq CM \rho^{-\gamma}.
 \]
 Since $r$ is defined as in \eqref{e:r}, and since $\alpha<1-\gamma$, we have that $r \leq \rho$. Then 
  \[
 \sup_{Q_{r/2}} |\partial_t u| \leq \sup_{Q_{\rho/2}} \leq  |\partial_t u| \leq CM \rho^{-\gamma}
 \leq CM r^{-\gamma},
 \]
 and we obtain \eqref{e:rho}. 
  For the H\"older bound for $u$, we utilize the $\alpha$-H\"older bound for $u_\rho$ (rather than than the $(1-\gamma)$-H\"older bound) as well as the uniform bound for $u_t$ just shown. 
 \[
 \begin{aligned}
 |u_t(x,t)-u_t(y,s)|&=Mr^{1-\gamma-1}
 |\partial_t u_\rho((y-x)/\rho,(t-s)/\rho^2)-\partial_t u_\rho(0,0)|\\
 &\leq C Mr^{-\gamma}(\sup_{Q_1} |\partial_t u_\rho|) \rho^{-\alpha}(|x-y|^{\alpha}+|t-s|^{\alpha/2}) \\
 &= C (\sup_{Q_\rho} |u_t|) \ \rho^{-\alpha-\gamma}(|x-y|^{\alpha}+|t-s|^{\alpha/2}) \\
 &\leq C \rho^{-\alpha-2\gamma} (|x-y|^{\alpha}+|t-s|^{\alpha/2}). 
 \end{aligned}
 \]
 Using that $r\leq \rho$ we obtain \eqref{e:rho2}. 
\end{proof}

\begin{theorem} \label{t:holder}
    The free boundary $\Gamma$ is a $C^{1,\alpha/2}$ function of the spatial variable $x$ for any $0<\alpha<1$. 
\end{theorem}

\begin{proof}
    Fix $0<\alpha<1$, and let $\alpha<\beta<1$. Fix $(x,t) \in \Gamma \cap Q_{1/2}$. If $|\nabla u(x,t)|=0$, then 
    $|\nabla u(y,s)|\leq |x-y|^{\alpha}+|t-s|^{\alpha/2}$, so that the angle $\theta$ between $\nu_{x,t}$ and $\nu_{y,s}$ satisfies
    \[
    \sin \theta \leq \frac{|\nabla u(y,s)|}{\sqrt{|\nabla u(y,s)|^2 + c^2}}
    \leq \frac{1}{c} (|x-y|^{\alpha/2}+|t-s|^{\alpha/2}). 
    \]
    We now assume that $|\nabla u(x,t)|>0$. 
    
   \noindent  \textbf{Case 1: $|x-y|+|t-s|^{1/2}\geq r/2$.}

     The angle $\theta$ between $\nu_{x,t}$ and $(0,1)$ is given by 
    \[
     \begin{aligned}
     \sin \theta_k&=\frac{|\nabla u(x,t)|}{\sqrt{|\nabla u(x,t)|^2 + (\partial_t u(x,t))^2}} \\
    &\leq \frac{|\nabla u(x,t)|}{c} \\
    &= \frac{Mr^{\beta}}{c}\\
    &\leq \frac{4M}{c}\left(|x - y|^{\beta}+|t-s|^{\beta/2} \right)\\
    &\leq \frac{4M}{c}\left(|x - y|^{\beta/2}+|t-s|^{\beta/2} \right) \\
    \end{aligned}
    \]
    Now
    \[
    \begin{aligned}
    |\nabla u(y,s)|&\leq |\nabla u(x,t)|+|\nabla u(x,t)-\nabla u(y,s)| \\
    &\leq Mr^{\beta} + |x-y|^{\beta}+ |t-s|^{\beta/2}\\
    &\leq 5M(|x-y|^{\beta/2}+ |t-s|^{\beta/2}).
    \end{aligned}
    \]
    Thus, a similar argument gives that the angle $\theta$ between $\nu_{y,s}$ and $(0,1)$ satisfies
    \[
    \sin \theta \leq \frac{10M}{c}\left(|x - y|^{\beta/2}+|t-s|^{\beta/2} \right).
    \]

   \noindent \textbf{Case 2: $|x-y|+|t-s|^{1/2}\leq r/2$.}
       To relate to the normal, we use the following vector valued function which projects onto the sphere. 
    \[
    F(\xi):= \frac{\xi}{|\xi|},
    \]
    for $\xi \in \mathbb{R}^{n+1}$. We will let $\xi_j =u_{x_j}(y,s)$ for $1\leq j \leq n$ and $\xi_{n+1}=u_t(y,s)$. If the differences $u_{x_i}(x,t)-u_{x_i}(y,s)$ and $u_t(x,t)-u_t(y,s)$ are small, then
    \[
     \nu_{x,t}-\nu_{y,s}\approx
     \begin{bmatrix}
         \frac{\partial F^1}{\partial \xi_1} & \dots &\frac{\partial F^1}{\partial \xi_n}
         & \frac{\partial F^1}{\partial \xi_{n+1}} \\
         \frac{\partial F^2}{\partial \xi_1} & \dots &\frac{\partial F^2}{\partial \xi_n}
         & \frac{\partial F^2}{\partial \xi_{n+1}} \\
          \vdots &\ddots &\vdots & \vdots\\
         \frac{\partial F^{n+1}}{\partial \xi_1} & \dots &\frac{\partial F^{n+1}}{\partial \xi_n}
         & \frac{\partial F^{n+1}}{\partial \xi_{n+1}} \\
     \end{bmatrix}
     \begin{bmatrix}
         u_{x_1}(x,t)-u_{x_1}(y,s) \\
         u_{x_2}(x,t)-u_{x_2}(y,s) \\
         \vdots \\
         \\
         u_{t}(x,t)-u_{t}(y,s) \\
     \end{bmatrix}.
    \]
    Now the difference $|u_{x_i}(x,t)-u_{x_i}(y,s)|$ will be small from H\"older continuity of the gradient, but $|u_t(x,t)-u_t(y,s)|$ is not necessarily small. However, we now show that we can still obtain a bound. First, 
    \[
    \frac{\partial F^j}{\partial \xi_k} = \frac{\delta_{kj}|\xi|^2 - \xi_k \xi_j}{|\xi|^{3/2}}. 
    \]
    Notice that $|\partial_{\xi_k} F^j|\leq 1/c$ since $|\eta| \geq c$. If $k \leq n$, then 
    $|u_{x_i}(x,t)-u_{x_i}(y,s)|$ is small, and 
    \[
     \left|\frac{\partial F^j}{\partial \xi_k} (u_{x_k}(x,t) -u_{x_k}(y,s))\right|
     \leq c^{-1}|(u_{x_k}(x,t) -u_{x_k}(y,s))| \leq c^{-1}(|x-y|^{\alpha} + |t-s|^{\alpha/2}). 
    \]
    If $k=j=n+1$, then $F^{n+1}$ is concave as a function of $\xi_{n+1}$. Then 
    \[
    |F^{n+1}(\xi_1,\ldots,\xi_n,u_t(x,t))- F^{n+1}(\xi_1,\ldots,\xi_n,u_t(y,s))|
    \leq \left|\frac{\partial F^{n+1}}{\partial \xi_{n+1}} (u_{t}(x,t) -u_{t}(y,s))\right|.
    \]
    Then from Corollary \ref{c:rescale} we have 
    \[
    \begin{aligned}
    \left|\frac{\partial F^{n+1}}{\partial \xi_{n+1}} (u_{t}(x,t) -u_{t}(y,s))\right|
    &\leq  \frac{|\xi|^2 - \xi_{n+1}^2}{|\xi|^{3/2}} |u_{t}(x,t) -u_{t}(y,s))|\\
    &\leq C (|x-y|^{\beta} + |t - s|^{\beta/2})^2 (|x-y|^{\alpha} + |t - s|^{\alpha/2})r^{-\alpha- 2\gamma} \\
    &\leq C (|x-y|^{\beta} + |t - s|^{\beta/2}). 
    \end{aligned}
    \]
    We now consider the final situation in which $k=n+1$ and $j<n+1$. By multiplying the original function $u$ by $c^{-1}$, we may assume $\partial_t u \geq 1$. This multiplication will only change the estimates by a factor of $c^{-1}$. Then  since $\partial_t u \geq 1$, then $F^j$ is a convex function of $\partial_t u$, and so 
    \[
    |F^{j}(\xi_1,\ldots,\xi_n,u_t(x,t))- F^{j}(\xi_1,\ldots,\xi_n,u_t(y,s))|
    \leq \left|\frac{\partial F^{j}}{\partial \xi_{n+1}} (u_{t}(x,t) -u_{t}(y,s))\right|.
    \]
     Then 
    \[
    \begin{aligned}
     \left|\frac{\partial F^{j}}{\partial \xi_{n+1}} (u_{t}(x,t) -u_{t}(y,s)) \right|
     &\leq Cr^{\beta} r^{-\gamma} r^{-\alpha-2\gamma}(|x-y|^{\alpha}+ |t - s|^{\alpha/2}) \\
    \end{aligned}
    \]
    We just have to let $1>\beta=\alpha+3\gamma$, and the result is proven.

\end{proof}

\section{Unbounded time derivative} \label{s:unbound}
In this section we show that Example \ref{ex:last} has an unbounded time derivative. Consequently, the space-time free boundary $\Gamma$ cannot be $C^{2,\alpha}$. We recall certain properties of Example \ref{ex:last}. We translate so that the origin is the last free boundary point in time. 
\begin{itemize}
    \item $u_t \geq 1$. 
    \item $\Gamma \cap \{|\nabla u|=0\}=(0,0)$. 
    \item $\Gamma \setminus (0,0)$ is locally a $C^{\infty}$ function of $x$. 
    \item $u(x,t)$ is radially symmetric in the spatial variable $x$. 
    \item For a fixed time $t<0$, we have $\{u(\ \cdot\ ,t)<0\}=B_r \times \{t\}$ for some $r$. 
\end{itemize}

\begin{lemma} \label{l:growth}
 Let $u$ be the constructed solution in Example \ref{ex:last} and translated, so that $\Gamma \cap \{|\nabla u|=0\}=(0,0)$. If $u_t$ is bounded in a neighborhood of $(0,0)$, then 
 \[
 S_r:= \sup_{Q_r} \frac{u(rx,r^2t)}{r^2}
 \]
 is bounded for $r\leq 1/2$. 
\end{lemma}

\begin{proof}
    Suppose by way of contradiction that 
    \[
     \sup_{Q_r} |u_t| \leq C,
    \]
     but there exists a sequence $r_j \to 0$ such that 
    \[
    \begin{aligned}
      &(1) \quad \frac{S_{r_j}}{r_j^2} =+\infty \\
      &(2) \quad S_{2^kr_j} \leq 2^{2k} S_{r_j}\quad \text{ for any integer $k$ with } k\leq j.  
    \end{aligned}
    \]
    We consider the sequence of functions 
    \[
    u_{r_k}:= \frac{u(r_kx, r_k^2 t)}{S_{r_k}},
    \]
    and note that there exists a limiting function $u_0$ satisfying 
    \begin{itemize}
     \item $u_0(0,0)=0$.
     \item $u_0(x,0)\geq 0$ if $|x|>0$. 
     \item $(\partial_t-\Delta) u_0 =0$ in $\mathbb{R}^n \times (-\infty,\infty)$,
     \item $\sup_{Q_R} |u_0| \leq C R^2$,
     \item $\partial_t u_0 \equiv 0$,
     \item $\sup_{Q_1} |u_0|=1$. 
    \end{itemize}
    From the third and fifth properties above, we have that $u_0$ is harmonic and time-independent. Then $u_0$ is nonnegative, harmonic on $\mathbb{R}^n$,   and $u_0(0)=0$. By the minimum principle,  we must have $u_0 \equiv 0$, but this will contradict the final property listed above. 
\end{proof}

If the solution constructed in Example \ref{ex:last} has a bounded derivative, then we can take a blow-up solution, i.e. there exists a sequence $r_k \to 0$ and a limiting function $u_0$ such that 
\[
u_{r_k}(x,t):=\frac{u(r_kx, r_k^2 t)}{r_k^2}
\]
converges to a limiting function (which we relabel $u$) and satisfying 
\begin{align}
 & \sup_{Q_R} |u| \leq C R^2 \text{ for } R>0. \label{e:growth} \\
 & u_t \geq 1. \\
 & (\partial_t - \Delta)u = \chi_{\{u>0\}} \quad \text{ in all } \mathbb{R}^n \times (-\infty,\infty). \label{e:solve}\\
 &\text{ for } t<0, \ \{u(\ \cdot\ , t)<0\}=B_{\rho(t)} \times \{t\}\quad \text{ for some $\rho$ depending on } t. \label{e:rhot}\\
 & u(x,t)=u(y,t) \text{ if } |x|=|y|, \label{e:symmetry} \\
 & \{u<0\} \subset \{t<0\} \label{e:below}. 
\end{align}

\begin{remark} \label{r:1bound}
 A similar argument to the proof of Lemma \ref{l:growth} will show that if one assumes $u(0,h)/h$ is unbounded for $h>0$, then necessarily $|u(0,h)/h|$ will be unbounded for $h<0$. 
\end{remark}

We now prove a Weiss-type monotonicity formula which is an adaptation of those found in both \cite{w99,aw06}. 

\begin{proposition} \label{p:weiss}
 Define $T_r:= \mathbb{R}^n \times(-4r^2,-r^2)$ and let 
 \[
 G(x,t):=\frac{1}{(4\pi (-t))^{n/2}} e^{\frac{-|x|^2}{-4t}} \quad \text{ the backwards heat kernel.}
 \]
 If 
 \[
 \Psi(r,u):=\frac{1}{r^4} \int_{T_r} \left(|\nabla u|^2- \max\{u,0\} +\frac{u^2}{t} \right)G(x,t) \ dx \ dt, 
 \]
 then 
 \[
 \frac{d}{dr} \Psi(r,u)=r^{-5}\int_{-4r^2}^{-r^2} \int_{\mathbb{R}^n} \frac{1}{-t} \left(2t u_t +  \langle \nabla u, x \rangle -2 u\right)^2 G(x,t) \ dx \ dt\geq 0. 
 \]
 Consequently, $\Psi(r,u)$ is nondecreasing, and $\lim_{r \to 0} \Psi(r)$ exists. If $\Psi(r,u)$ is constant, then $u(rx,r^2t)=r^2u(x,t)$. 
\end{proposition}

\begin{proof}
 We note that by the growth condition \eqref{e:growth}, the integrals are well-defined, and all integration by parts will be justified. The only other condition we will use is \eqref{e:solve}. 
 If $u_r(x,t):=r^{-2}u(rx,r^2t)$, then by scaling $\Psi(r,u)=\Psi(1,u_r)$. We note that 
 \[
 \frac{d}{dr} u_r(x,t) = r^{-1}\left(-2u_r(x,t) + \nabla u_r(x,t) \cdot x+2 t u_t  \right). 
 \]
 Then 
 \[
 \begin{aligned}
     \frac{d}{dr} \Psi(r,u) 
     &= \frac{d}{dr} \Psi(1,u_r)\\
     &=\frac{d}{dr} \int_{-4}^{-1} \int_{\mathbb{R}^n} \left(|\nabla u_r|^2- 2\max\{u_r,0\} +\frac{u_r^2}{t} \right)G(x,t) \ dx \ dt\\
     &= \int_{-4}^{-1} \int_{\mathbb{R}^n} \frac{d}{dr}\left(|\nabla u_r|^2- 2\max\{u_r,0\} +\frac{u_r^2}{t} \right)G(x,t) \ dx \ dt\\
     &= \int_{-4}^{-1} \int_{\mathbb{R}^n} \left(2\langle\nabla u_r,\nabla \frac{d u_r}{dr}\rangle- 2\chi_{\{u_r>0\}}\frac{d u_r}{dr} + 2 \frac{u_r}{t} \frac{d u_r}{dr}\right)G(x,t) \ dx \ dt\\
     &= 2\int_{-4}^{-1} \int_{\mathbb{R}^n} \left(-\Delta u_r-\chi_{\{u_r>0\}} + \frac{u_r}{t}\right)\frac{d u_r}{dr}G(x,t) - \langle \nabla u_r ,\nabla G(x,t)\rangle \frac{d u_r}{dr}\ dx \ dt \\
     &=2\int_{-4}^{-1} \int_{\mathbb{R}^n} \left(-\partial_t u_r + \frac{u_r}{t}\right)\frac{d u_r}{dr}G(x,t) - \langle \nabla u_r ,\nabla G(x,t)\rangle \frac{d u_r}{dr}\ dx \ dt \\
     &=2\int_{-4}^{-1} \int_{\mathbb{R}^n} \left(-\partial_t u_r +\frac{u_r}{t}\right)\frac{d u_r}{dr}G(x,t) - \frac{1}{2t}\langle \nabla u_r ,x\rangle G(x,t)\frac{d u_r}{dr}\ dx \ dt \\
      &=2\int_{-4}^{-1} \int_{\mathbb{R}^n} \left(-\partial_t u_r + \frac{u_r}{t}- \frac{1}{2t}\langle \nabla u_r ,x\rangle\right)\frac{d u_r}{dr}G(x,t) \ dx \ dt \\
      &= \int_{-4}^{-1} \int_{\mathbb{R}^n} \frac{r}{-t} \left(\frac{d u_r}{dr} \right)^2 G(x,t) \ dx \ dt \\
      &\geq 0. 
 \end{aligned}
 \]
 Rescaling back, we obtain the representation $\frac{d}{dr}\Psi(r,u) $. The case of equality holds if and only if $\frac{d}{dr} u_r =0$ for all $r$, or if and only if $u_r$ is constant in $r$, or if and only if $u(rx,r^2t)=r^2u(x,t)$. 
\end{proof}

\begin{lemma} \label{l:blowup2}
 Let $u$ satisfy the conditions above. For any sequence $r_k \to 0$, there exists a subsequence and a limiting function $u_0$ such that $\lim u_{r_k} = u_0$ and $u_0$ also satisfies the conditions \eqref{e:growth} through \eqref{e:below}, and $u_0(rx,r^2t)=r^2u(x,t)$ for all $x$ and $t<0$. 
\end{lemma}

\begin{proof}
 That $u_0$ exists and satisfies all the conditions is straightforward. Now 
 \[
 \Psi(\rho,u_0)= \lim_{r_k \to 0} \Psi(\rho, u_{r_k})= \lim_{r_k \to 0} \Psi (\rho r_k, u)= 
 \Psi(0+,u). 
 \]
 Thus, $\Psi(\rho, u_0)$ is constant, and so $u_0(rx,r^2t)=r^2u_0(x,t)$. 
\end{proof}

\begin{theorem} \label{t:blowup}
 If dimension $n=1$, there does not exist a solution satisfying \eqref{e:growth} through \eqref{e:below} as well as $u(rx,r^2t)=r^2u(x,t)$ for $t<0$. 
\end{theorem}

\begin{proof}
    From the homogeneity property, for $t<0$ the function $u(x,t)$ is uniquely defined by $u(x,-1)$ in the following way $u(x,t)=-tu(x/\sqrt{-t},-1)$.
    Now \[
    u_t = -u(x/\sqrt{-t},-1) + \frac{1}{2}\frac{x}{\sqrt{-t}}u(x/\sqrt{-t},-1),
    \]
    and $u_{xx}(x,t)=u_{xx}(x/\sqrt{-t},-1)$. 
    %For small $h>0$ we have that 
    %\[
    %u(x,-1+h)=u(\sqrt{1-h}\frac{x}{\sqrt{1-h}},-(\sqrt{1-h})^2)
    %=(1-h)u(x/\sqrt{1-h},-1). 
    %\]
    %Thus, 
    %\[
    %\begin{aligned}
    %u_t(x,-1)&= \lim_{h \to 0} \frac{u(x,-1+h)-u(x,-1)}{h} \\
    %         &= \lim_{h \to 0} \frac{(1-h)u(x/\sqrt{1-h},-1) - u(x,-1)}{h}\\
    %         &= \lim_{h \to 0} \frac{u(x/\sqrt{1-h},-1)}{h} - u(x,-1)\\
    %         &= \lim_{h \to 0} u_{x}(x/\sqrt{1-h},-1)(-1/2)x(1-h)^{-3/2}(-1)  - u(x,-1)\\
    %         &= \frac{1}{2}x u_x(x,-1) - u(x,-1). 
    %\end{aligned}
    %\]
    Now considering the set on which $\{u<0\}$ we have $u_t -  u_{xx}=0$ which implies that 
    \[
    \frac{1}{2}x u_x(x,-1) - u(x,-1) - u_{xx}(x,-1)=0. 
    \]
    From \eqref{e:symmetry}, we have $u_x(0,-1)=0$. 
    Thus, if $f(x)=u(x,-1)$, then 
    \[
    \begin{cases}
        -f''(x)+\frac{1}{2}xf'(x)-f(x)=0 \\
        f(0)=-c<0, \quad f'(0)=0.
    \end{cases}
    \]
    The above solution is unique and is given by $f(x)=c(-1+\frac{1}{2}x^2)$. Then $u(x,t)=c(t+\frac{1}{2}x^2)$ on $\{u<0\}$. On the set $\{u>0\}\cap \{t<0\}$, we must also have 
    \[
     \begin{cases}
        -f''(x)+\frac{1}{2}xf'(x)-f(x)=1 \\
        f(\sqrt{2})=0, \quad f'(\sqrt{2})=c,
    \end{cases}
    \]
    and the additional requirement that $f(x)\geq 0$ for $x\geq \sqrt{2}$ from \eqref{e:rhot}. 
    The solution is unique, and from the theory of series solutions for ODEs we have that if 
    \[
    f(x)=\sum_{n=0}^{\infty} a_n (x-\sqrt{2})^n, 
    \]
    then we have the coefficients
    \[
     a_0=0, \quad a_1=c, \quad a_2 = \frac{\sqrt{2}}{4}c - \frac{1}{2},
    \]
    and the recursion relation
    \begin{equation} \label{e:recursion}
    \frac{\left(\frac{n}{2}-1 \right)a_n +\frac{\sqrt{2}}{2}(n+1)a_{n+1}}{(n+2)(n+1)}=a_{n+2} \quad \text{ for } n \geq 1. 
    \end{equation}
    Then 
    \[
    a_3= -\frac{\sqrt{2}}{12},
    \]
    and notice this is independent of $c$. Then
    \[
    a_4 = \frac{0 \cdot a_2 + \frac{\sqrt{2}}{2}(3)a_3}{4\cdot 3}=-\frac{1}{48}. 
    \]
    Since both $a_3,a_4 <0$, it follows from the recursion relation \eqref{e:recursion} that $a_n <0$ for $n \geq 3$. Then for any $c>0$, there is an $x_0$ large enough so that $f(x)<0$ for $x>x_0$. We then obtain a contradiction. 
\end{proof}

%\begin{corollary} \label{c:not2}
%    Let $n=1$ and let $u$ the constructed solution in Example \ref{ex:last}. Then $\Gamma \notin %C^{2,\gamma}$ for any $\gamma>0$. 
%\end{corollary}

%\begin{proof}
%    From Theorem \ref{t:blowup} we have that $|u(0,h)/h|$ is unbounded. From Remark %\ref{r:1bound} it follows that $|u(0,h)/h|$ is unbounded for $h<0$. Suppose now that $\Gamma \in %C^{2,\gamma}$ for some $\gamma>0$. Now $\Gamma''(0)\leq 0$ since $u(x,0)\geq 0$. 
%
%    \noindent \textbf{Case 1: $\Gamma''(x)<0$.}
%    
%    Let $-a=\Gamma''(0)$. Consider $v(x,t)=2at+x^2$ which is a supersolution, and thus lies %above. It has bounded time derivative. A contradiction. 
%    
%    Then $\Gamma$ can be touched by above at $0$ by a paraboloid.     
%\end{proof}

\bibliographystyle{amsplain}
\bibliography{unstablereferences}

\end{document}